\documentclass[11pt,a4paper]{article}

\usepackage[a4paper,top=3.8cm,bottom=3.8cm,left=3cm,right=3cm]{geometry}%
\usepackage{hyperref}

\usepackage[T1]{fontenc}
\usepackage{amsmath,amssymb,mathtools,amsthm,amsopn}
\usepackage{mathpazo}
\usepackage{graphicx}  
\usepackage[usenames,dvipsnames]{xcolor}

\usepackage{yfonts}
\usepackage{tocloft}
\usepackage{fancyhdr}
\usepackage{fourier-orns}

\usepackage{enumitem}

\usepackage[dayofweek]{datetime}
\usepackage{mathrsfs}
\renewcommand{\phi}{\varphi}

\newcommand{\norm}[1]{\left\Vert#1\right\Vert}

\newcommand{\expap}{\exp_{\textup{ap}}}

\newcommand{\abs}[1]{\lvert#1\rvert}

\newcommand{\g}{\gamma}

\renewcommand{\L}{{\mathcal{L}}}
\newcommand{\wh}{\widehat}

\newcommand{\e}{\varepsilon}
\renewcommand{\r}{\rho}

\newcommand{\s}{\sigma}

\newcommand{\wt}{\widetilde}

\newcommand{\ol}{\overline}

\renewcommand{\d}{\delta}
\newcommand{\Eucl}{\textup{Euc}}
 
\newcommand{\p}{\partial}

\newtheoremstyle{pippo}  
  {}       
  {}       
   {\sffamily}   
 {}        
  {\bfseries}  
  {.}   
  {1ex}       
  {}           
\newtheoremstyle{pluto}  {}{}
{\slshape}  {}{\bfseries}  {.} {1ex}    {}

\newtheorem{theorem}{Theorem}[section]

\newtheorem{lemma}[theorem]{Lemma}

\theoremstyle{pluto}
\newtheorem{definition}[theorem]{Definition}
\newtheorem{remark}[theorem]{Remark}

\usepackage{mdwlist}

\newcommand{\R}{\mathbb{R}}

\newcommand{\N}{\mathbb{N}}
\newcommand{\C}{\mathbb{C}}

\renewcommand{\d}{\delta}
\renewcommand{\t}{\tau}

\renewcommand{\a}{\alpha}
\renewcommand{\b}{\beta}

\DeclareMathOperator{\Lip}{Lip}
\DeclareMathOperator{\Span}{span}

\numberwithin{equation}{section}

\let\oldbibliography\thebibliography
\renewcommand{\thebibliography}[1]{%
  \oldbibliography{#1}%
  \setlength{\itemsep}{0pt}%
}

\usepackage{titlesec}
\titleformat{\section}{%
\normalfont\large\bfseries}{\thesection.}{1em}{}
\titleformat{\subsection}{%
\normalfont\normalsize\bfseries}{\thesubsection.}{1em}{}

\begin{document}

\title{Anisotropic estimates of subelliptic  type
 }
\author{Annamaria Montanari \and Daniele Morbidelli}

\date{}

\maketitle


\begin{abstract}
We discuss some  estimates of subelliptic type related with vector fields satisfying the H\"ormander condition. 
Our approach makes use of a class of approximate exponentials studied in our previous papers \cite{Morbidelli98,LanconelliMorbidelli,MM04,MM}. 
Such kind of estimates   arises naturally in the study of  regularity theory of weak solutions of degenerate elliptic equations. 
\end{abstract}


\section{Introduction}
In this  note we {  review and slightly improve \color{black} some
  estimates of subelliptic type  for a family  $X_1,\dots, X_m$ of smooth vector fields of H\"ormander type in $\R^n$.  We mainly use the analysis of a class of approximate exponential maps appearing in some previous papers   of the authors and of Ermanno Lanconelli.  See \cite{Morbidelli98,LanconelliMorbidelli,MM04,MM}. We
shall formulate our estimates making use of a family  of fractional Sobolev norms modeled on the subRiemannian geometry defined by $X_1, \dots, X_m$. These norms have been analyzed in \cite{Morbidelli98}.

Let  $X_1, \dots, X_m$ be a family of vector fields satisfying  the H\"ormander condition of step $\kappa\in\N$ in~$\R^n$. A classical estimate 
 states that given a bounded set $\Omega\subset\R^n$ and $p\in\left[1,+\infty\right[$, 
 there is a positive constant $C$ such that
\begin{equation}\label{classica} 
 [f] _{W^{ 1/ k ,p }}:=\bigg(\int_{\Omega\times\Omega}
 \frac{|f(x)-f(y)|^p}{|x-y|^{n+p(1/\kappa) }}
 dxdy\bigg)^{1/p}
 \leq C \Big(\|f\|_{L^p}+\sum_j\|X_j f\|_{L^p}\Big), 
\end{equation} 
for all smooth $f$ compactly supported in $\Omega$. Here $[f]_{W^{ 1/ \kappa ,p }}$ denotes the $L^p $ (semi)norm of the fractional 
derivative of order $1/\kappa  $ of $f$.
Classical references for this inequality are  \cite{RothschildStein,FeffermanPhong81}.    Various versions of estimate~\eqref{classica} have been used extensively starting from the seminal paper of H\"ormander \cite{Hormander67} in the theory of  hypoelliptic operators.

Estimate~\eqref{classica} inherently gives  the 
same order~$1/\kappa$ of  Euclidean  fractional differentiability
in all directions, for a function $f\in L^p$ with derivatives $X_1 f , \dots, X_mf\in L^p$. Note that in typical situations of interest in subelliptic analysis,  the dimension of the subspace generated by  $ X_1(x),\dots, X_m(x) $ is strictly less than the topological dimension~$n$.  
It is known that~\eqref{classica} is sharp, as far as we refer to Euclidean fractional derivatives. However, it does not capture the fact that given $f$ in the first order Folland--Stein space,\footnote{The  functional space defined by the norm in the right-hand side of~\eqref{classica}.}  one can expect a better regularity of~$f$ along the directions of commutators of lower order compared with the regularity expected  along the directions of commutators of higher order. This motivates the introduction of a different notion of fractional differentiability, based   on the data of the metric measure space $(\R^n,\L^n,d)$, where $\L^n$ denotes the Lebesgue measure in~$\R^n$ and $d$  \color{black} the subRiemannian distance. Namely, following \cite{Morbidelli98}, given $\Omega\subset\R^n$, we define for any $s\in\left]0,1\right[$ and $p\in\left[1,+\infty\right[ $,  the seminorm
\begin{equation}\label{semin} 
 [f ]_{W^{s,p}_d(\Omega )} :=\bigg(\int_{\Omega\times\Omega}\frac{|f(x)-f(y)|^p}{\L^n(B(x,d(x,y)))d(x,y)^{ps}}dxdy\bigg)^{1/p},
\end{equation} 
where $d$ denotes the subRiemannian distance associated with the vector fields~$X_1,\dots, X_m$. Here  and hereafter,
$B(x,r)$ will denote the ball  of center $x\in\R^n$ and radius $r$ with respect to the subRiemannian distance~$d$.  Observe that, by known      properties in subRiemannian geometry, we have the  local estimates $\L^n(B(x, r))\leq C r^n$ and $d(x,y)\leq C|x-y|^{1/\kappa}$, where $\kappa$ is the step of the vector fields. See~\eqref{upside2} in Section~\ref{Sec2}.  Then 
 we have the trivial local embedding  
 \begin{equation*}
\begin{aligned}
  \|f\|_{L^p(\Omega)}&+\Big( 
\int_{\Omega\times\Omega}\frac{|f(x)-f(y)|^p}{|x-y|^{n+p(s/\kappa)}}dxdy\Big)^{1/p}  
\leq C \|f\|_{W^{s,p}_d(\Omega)},
\end{aligned}
 \end{equation*} 
where $\|f\|_{W^{s,p}_d(\Omega)} :=  \norm{f}_{L^p(\Omega)}+    [f ]_{W^{s,p}_d(\Omega )}$ and the positive constant $C$ depends on the bounded set
$\Omega\subset\R^n$   and on $s\in\left]0,1\right[$.

Note also that  the embedding is somehow strict.        In order to explain our comment,  given a bounded set $\Omega$, we define  
the fractional derivative of order $\e \in\left]0,1\right[$ of a function $f$ along a vector field $Z$ as   
\begin{equation}
 [f ]_{W^{\e,p}_Z(\Omega)} :=\bigg(\int_\Omega dx 
  \int_{\{t\in[0,1]\,:\, e^{tZ}(x)\in\Omega\}} 
 \frac{dt}{|t|^{1+p\e}}
 |f(e^{tZ}x)-f(x)|^p  \bigg)^{1/p},
\end{equation}  
 where as usual we denote by $e^{tZ}(x)$  the value at time $t$ of the integral curve of $Z$ starting from $x$ when $t=0$.   

Let us introduce the 
notation $X_w:=[X_{w_1},\dots,[X_{w_{\ell-1}}, X_{w_\ell} ]\dots]$ to denote nested commutators of length $|w|=\ell \leq\kappa$. Then,   
  we  show that given vector fields of step~$\kappa\in\N$, $p\in\left[1,+\infty\right[$ and a triple   $\Omega_1\Subset \Omega_2\Subset \Omega_3$ of  bounded sets, 
  for all $s\in\left]0,1\right[$ we have the equivalence
\begin{equation}\label{equivalgo} 
 C^{-1}\|f\|_{W^{s,p}_d(\Omega_1)}\leq    
  \|u\|_{L^p(\Omega_2)}+\sum_{|w|\le\kappa}[f]_{W^{s/|w|,p}_{X_w}(\Omega_2)}
 \leq C
 \|f\|_{W^{s,p}_d(\Omega_3)}.
\end{equation}
This means that given a function $f\in  L^p(\Omega)$ and $s<1$, the seminorm $ [f]_{W^{s,p}_d(\Omega)}$  is finite if and only if $f$ has $s/|w|$ derivatives in $L^p$ along commutators~$X_w$ of length $|w|= 
1,2,3,\dots, \kappa$. Note that in \cite{Morbidelli98} the first author proved the equivalence~\eqref{equivalgo} for commutators $ X_w $ with length~$|w|=1$ only. Here we    show that the argument in \cite{Morbidelli98}  provides also    the inequality 
$$
[f]_{W^{s/|w|,p}_{X_w}(\Omega_2)}
 \leq C
 \|f\|_{W^{s,p}_d(\Omega_3)}
$$
  for  commutators $X_w$ of arbitrary length $|w|\leq \kappa.$

 In Section \ref{Sec3} we shall prove the following anisotropic subelliptic estimate. 
 \begin{theorem}\label{anisotta}
Given H\"ormander vector fields $X_1,\dots, X_m$ in $\R^n$,
for all $p\in\left[1,+\infty\right[$, for all $s \in\left]0,1 \right[$ and for any  pair of nested bounded sets $\Omega\Subset \Omega_0$ there is $C>0$ such that for all $C^1$ function $f$ we have the inequality
\begin{equation}\label{proxx} 
   [f]_{W^{s,p}_d(\Omega)}\leq C \Big(\sum_{j=1}^m \| X_j f\|_{L^p(\Omega_0)} 
   +\|f\|_{L^p(\Omega_0)}\Big).
\end{equation}   \end{theorem}
Roughly speaking, the inequality \eqref{proxx} means that  the $L^p$ norm of the derivatives of order $s<1$ in the subRiemannian   metric measure  space $(\R^n,\L^n, d)$      can be estimated from above with the derivatives of order $1$ in the Folland--Stein space. 
This estimate is trivial  in the Euclidean case, see the discussion in 
Remark~\ref{oca}.  
Surprisingly, the proof of the subRiemannian statement \eqref{proxx} becomes less trivial and requires a certain amount of work.

Inequality \eqref{proxx} has been proved in \cite[Theorem 5.1]{Morbidelli98} for $1<p<\frac{Q}{1-s}$ and for  an  appropriate $Q,$
by using some nice properties of the fundamental solution of H\"ormander operators. Here we extend this inequality to every $p\geq 1$ by using a completely different technique, whose main tool is the approximate exponential map.
 The proof of Theorem~\ref{anisotta} will be presented in Section \ref{Sec3} and it is   inspired to the argument  of the Lanconelli's unpublished proof 
 of the nonsharp version  of \eqref{classica}.  See  \cite[Proposition~6.2]{MM}.
 Let us mention that for the case $s=1$ our seminorm~\eqref{semin}
is not useful,  but there is a rich theory of Sobolev spaces   of order~$s=1$    defined on metric measure spaces. See the   Haj{\l}asz spaces \cite{Hajlasz} and the Newtonian spaces \cite{Shanmugalingam}, just to quote a few.

 In the subsequent Section \ref{Sec4}, following~\cite{Morbidelli98},
 we describe the proof of the equivalence~\eqref{equivalgo}. As a corollary, we obtain estimates in the  directions  of commutators.

 Namely, we   will     get  the following statement.
 \begin{theorem}\label{parolone} 
Let $X_1, \dots, X_m$ be H\"ormander vector fields of step $\kappa$ in $\R^n$ and take $p\in\left[1,+\infty\right[$.
Let $s\in\left]0,1\right[$ and consider a  commutator  $X_w$ of length $|w|\leq \kappa$.  Then, given  bounded  open sets $\Omega\Subset \Omega_0$ there is  $C>0$ such that 
  \begin{equation*}
\begin{aligned}
   [f]_{W^{s/|w| ,p}_{X_w} (\Omega)}: = &\bigg(
   \int_\Omega dx \int_{\{t\in[0,1]\,:\,e^{tX_w}(x)\in\Omega\}}  \frac{dt}{|t|^{1+ps/|w|}}
  \big|f(e^{tX_w}x)-f(x)\big|^p dx\bigg)^{1/p}
\\&\leq C\Big(\|f\|_{L^{p}(  \Omega_0)} +\sum_{j=1}^m \|X_j f\|_{L^p(  \Omega_0)} \Big).
\end{aligned}
  \end{equation*} 
 \end{theorem}
 The idea of using anisotropic estimates along 
 different directions in a subelliptic 
 context   arises naturally in the study of  pointwise estimates for  weak solutions of  degenerate elliptic equations with measurable coefficients and it was exploited long ago 
 by Franchi and Lanconelli~\cite{FranchiLanconelli84} in the setting of the diagonal vector fields $X_j=\lambda_j\frac{\p}{\p x_j}$, where $j=1,\dots, n$ and $\lambda_1,\dots, \lambda_n$ are suitable functions. 
  Here we formulate a family of anisotropic  inequalities in the setting of H\"ormander vector fields with their commutators. Again, our techniques do not provide a proof of the borderline case~$s=1$.

 To motivate  Theorem \ref{parolone}, let us give a formulation of its  $L^\infty$  version, starting from  a well known fact. Let $X=\p_x+2y\p_t$ anf $Y=\p_y-2x\p_t$ be the vector fields of the Heisenberg group with coordinates $(x,y,t)$. Then, by an easy computation,  we have the exact formula 
\begin{equation}\label{fox} 
 e^{-sY} e^{-sX} e^{sY} e^{s X}(x,y,t)= (x,y, t-4s^2)= e^{s^2 [X,Y]}(x,y,t),
\end{equation} 
for all $(z,t):=(x,y,t)\in\R^3$ and $s>0$. Then, given any regular function $f=f(z,t)$, we have an estimate  of the $\frac 12$-H\"older seminorm of $f$ on a bounded set $\Omega$:
\begin{equation}\label{jolog} 
 \sup_{(z,t)\in\Omega,\; |\t|\leq r_0}  
 \frac{\big| f(e^{\t[X,Y]} (z,t))-f(z,t)\big|}{|\t|^{1/2}}
 \leq C \sup_{(z,t)\in  \Omega_0}\big(|Xf(z,t) |+|Yf(z,t)|\big),
\end{equation} 
where $ \Omega_0\supset\ol\Omega$ and $r_0>0$ is small enough. 

 A natural generalization of  \eqref{fox} to more general  vector fields $X_1,\dots, X_m$ would be   a formula of the following form.
 Let $X_w$  be any commutator of length $|w|\leq\kappa$ constructed from a family $X_1,\dots, X_m$ of H\"ormander vector fields of step~$\kappa$ in $\R^n$.
 Given  a bounded set $\Omega$  
 and $\Omega_0\supset\ol\Omega$ \color{black} there is  $r_0>0$ and  a positive $C$ such that for all nested commutator $X_w:=[X_{w_1},\dots [X_{w_{\ell-1}}, X_{w_\ell}]\dots ]$ of length $|w|\in\{1,2,\dots, \kappa\}$,
we have 
\begin{equation}\label{kiki} 
 \sup_{x\in\Omega ,\;|\tau|\leq r_0} \frac{|f(e^{\tau  X_w}x) - f(x)|}{|\tau|^{1/|w|}}\leq C \sup_{y\in  \Omega_0} \sum_{j=1}^m|X_j f(y)|.
\end{equation} 
If we try to prove \eqref{kiki}  by generalizations of formula~\eqref{fox}, we encounter some remainders  which can not be controlled with elementary methods.
However, estimate~\eqref{kiki} does hold as a consequence of the ball-box Theorem presented in Section~\ref{Sec2}. See the explanation in 
Remark~\ref{march}.

We conclude the Introduction by  remarking that in this paper, for the sake of clarity, we consider smooth vector fields satisfying H\"ormander condition. 
However,
 by the ball-box Theorem in \cite{MM}, we expect that all the results of the present paper can be extended to 
 nonsmooth vector fields of arbitrary step $\kappa\in \N$     and with coefficients in some regularity class related with the step $\kappa$ 
  appearing in the H\"ormander condition.

\section{Preliminaries}\label{Sec2}

Consider  smooth vector fields $  X_1, \dots, X_m $ in $\R^n$.
Given a word $w=w_1\cdots w_\ell$ in the alphabeth $\{1,\dots,m\}$, let us introduce the commutator
\begin{equation*}
X_w := [X_{w_1},         [X_{w_2}, \dots [X_{w_{\ell-1}}, X_{w_\ell}]\dots ].         
\end{equation*}
The number $|w|=|w_1w_2\cdots w_\ell| =:\ell $ is called the \emph{length} of the commutator $X_w$.
Let us define the \emph{subRiemannian distance}
\begin{equation*}
\begin{aligned}
d (x,y) &:= \inf \big\{r>0: \text{ there is $\gamma\in
\Lip((0,1),\R^n)$
with $\g(0)=x, \g(1)= y$}
\\&
 \qquad\qquad  \qquad\text{and $\dot\gamma(t)= \sum_{1\le j\le m} u_j(t) rX_j  (\gamma(t))$ with $ |u(t)|_{\Eucl}\le
1 $ 
for a.e. $t\in [0,1]$} \big\}.
\end{aligned}
\end{equation*}

Given a fixed $\kappa \ge 1$, denote by
$    Y_1, \dots, Y_q $ an enumeration of $  \{ X_w: 1\le \abs{w} \le \kappa\}$,
the family of commutators of length at most $\kappa$.  Let
$  \ell_j\le \kappa$  be the length of $Y_j$. Define  
the distance $\r$ 
\begin{equation}\label{coscos}
\begin{aligned}
&\text{$\r(x,y)   := \inf \big\{  r\ge 0 :   $ there is $\gamma\in
\Lip((0,1), \R^n)$ such that $\gamma(0)=x$   }
\\&\quad \text
{
$\gamma(1) = y$ and $\dot\gamma(t)= \sum_{j=1}^q  b_j(t) r^{\ell_j}Y_j(\g(t)): \abs{b(t)}_\Eucl\le 1 $ for a.e. $t\in
[0,1]$}\big\}.
\end{aligned}
\end{equation}
The H\"ormander condition of step~$\kappa$ reads as 
$\Span\{X_w(x) : |w|\leq \kappa \}=\R^n$ for all 
$x\in\R^n$.

We denote by  $B_\r(x,r)$, $B (x,r)$ and $B_\Eucl(x,r)$ the balls of center $x$ and radius $r$ with respect to $\r$, $d$ and the Euclidean distance
respectively.   Sometimes, to avoid any confusion, we denote by $|\cdot|_\Eucl$ the Euclidean norm.

Since the vector fields $Y_1,\dots, Y_q$ span $\R^n$ at any point,  it is easy to see that  for all pair of points $x,y\in\R^n$, the set of competitors defining   $\rho(x,y)$ is nonempty and then $\rho(x,y)<+\infty$ for any pair of points. Furthermore, an elementary argument shows that a local estimate of the form  
$
\rho(x,y)\leq C|x-y|^{1/\kappa}                                                                                                                                                                                                                                                                                                                                            $   holds.   Trivially,  by definition, we have $\rho(x,y)\leq d(x,y)$.
The
classical Chow's theorem implies also that  for any pair of points $x$ and $y\in\R^n$ 
the set of competitors defining $d(x,y)$ is nonempty. Then   $d(x,y)<\infty$. Finally, both   $\rho$ and $d $ satisfy the axioms of  a distance.

\paragraph{Ball-box Theorem}
We recapitulate here the statement of the ball-box Theorem. 
Let us consider a family of vector fields $X_1, \dots, X_m$ satisfying the H\"ormander
  condition   of step $\kappa$ in $\R^n$.   
Let us fix an enumeration $Y_1,\dots, Y_q$ of all the commutators 
$X_w$ with  length $|w|\leq \kappa$.
Let  $\ell_i$ be the length of $Y_i.$  If the H\"ormander condition of step $\kappa$
 is fulfilled, then the vector fields $Y_1,\dots, Y_q$ span $\R^n$ at any point.  
Given a multi-index $I=(i_1,\dots,i_n)\in \{ 1,\dots,q\}^n$  and its corresponding  $n-$tuple  $Y_{i_1}, \dots ,Y_{i_n}$, let
\begin{equation}\label{massi}
 \begin{aligned}\lambda_I(x)& =\det (Y_{i_1}(x), \dots ,Y_{i_n}(x)) ,
\quad\text{and}
\quad \ell(I)& = \ell_{i_1}+\cdots +\ell_{i_n}.
\end{aligned} \end{equation}
The first ``ball-box Theorem'' was proved by Nagel Stein and Wainger in  \cite{NagelSteinWainger}. Namely, in that paper, the authors introduced the exponential map related with an $n$-tuple   $I=(i_1, \dots, i_n)\in\{1,\dots,q\}^n$ in the form  
\begin{equation*}
 \Phi_{I,x}(u):=\exp\Big(\sum_{j=1}^n u_j Y_{i_j}\Big)(x),
\end{equation*} 
where $u$ belongs to a neighborhood of the origin in $\R^n$ and $\exp(Z) (x)$ or $e^Z x$ denotes  the value at time $t=1$ of the integral curve of the vector field $Z$ starting from $x\in\R^n$ at $t=0$. Since we are interested in local estimates, we may without loss of generality assume that $e^Z x$ is well defined   
in all situations of our interest.  

To give the statement of the Nagel--Stein--Wainger ball-box Theorem, we introduce a definition. 
 \begin{definition}[$\eta$-maximal triple] 
 Let $I=(i_1, \dots, i_n)\in\{1,\dots,q\}^n$, $x\in\R^n$, $\eta\in\left]0,1\right[$ and $r>0$.  
 We say that $(I,x,r)$ is $\eta$-maximal if
\begin{equation}\label{maxx}
 |\lambda_I(x)| r^{\ell(I)}>\eta \max_{ K \in\{1,\dots,q\}^n} |\lambda_ K (x)| r^{\ell(K)}.
\end{equation}
 \end{definition}
 Define also for all $I\in\{1,\dots, q\}^n$
 \begin{equation}\label{qui} 
\|h\|_I=\max_{j=1,\dots,n}|h_j|^{1/\ell_{i_j} },\qquad   Q_I(r)=\{h\in\R^n:\|h\|_I < r\}. 
\end{equation}
\color{black}
Then, in \cite{NagelSteinWainger} the authors
proved that if $(I,x,r)$ is $\eta$-maximal, $x$ belongs to a compact set and $r$ is sufficiently small, then we have the double inclusion
\begin{equation}
 \Phi_{I,x}(Q_I(c_1 r))\subset B_\rho(x,r)\subset \Phi_{I,x}(Q_I(c_2r)),
\end{equation} 
where the constants $c_1$ and $c_2$ depend on $\eta$ and on the compact set where $x$ lies.  This estimate together with some Jacobian estimates have the consequence that 
 \begin{equation}\label{misura} 
  \L^n(B_\rho(x,r))\simeq \sum_{K\in\{1,\dots, q\}^n}|\lambda_K(x)| r^{\ell(K)},
 \end{equation} 
where the equivalence holds for compact sets and sufficiently small $r$.   We have presented the statements of this part in an informal way. Below we shall give precise statements of similar results which are needed in this paper. \color{black}

After \cite{NagelSteinWainger}, the analysis of the maps $\Phi_I$ was carried out by several authors in subsequent years. See \cite{BramantiBrandoliniPedroni13,TaoWright03,Street}. Since $\ell(I)\geq n$ for all $I$, we always have by~\eqref{misura} the local estimate
$ \L^n(B_\rho(x,r)) \leq C r^n,
$
where  $x$ belongs to a bounded set and $r>0$ is sufficiently small.

 \paragraph{Approximate exponentials of commutators and    the corresponding ball-box Theorem.} 
 In \cite{NagelSteinWainger}, the authors gave also a sketch of the proof of the fact 
 that the distance $\rho$ is equivalent to $d$ locally. This was done introducing a class of maps which we are now going to   call   
 ``approximate exponentials''.  

Consider  vector fields $X_{w_1}, \dots, X_{w_\ell}$, and their commutator $X_w$,   which has length $\ell=|w|$. \color{black} Let us define the approximate exponential $\expap(tX_w) $. For $\t>0,$ we define, as in \cite{NagelSteinWainger}, \cite{Morbidelli98} and \cite{MM},
\[
 \begin{aligned}
 C_\t( X_{w_1})&:= \exp(\t X_{w_1}),
 \\ C_\t( X_{w_1}, X_{w_2})&: =\exp(-\t X_{w_2})\exp(-\t X_{w_1})\exp(\t X_{w_2})\exp(\t X_{w_1}),
 \\&\vdots
  \\C_\t( X_{w_1}, \dots, X_{w_\ell})&
:=C_\t( X_{w_2}, \dots, X_{w_\ell})^{-1}\exp(-\t X_{w_1}) C_\t( X_{w_2}, \dots, X_{w_\ell})\exp(\t X_{w_1}). \end{aligned}
 \]
Then let
\begin{equation}
 \expap(t X_w): =\left\{\begin{aligned}
& C_{t^{1/\ell}}(X_{w_1}, \dots, X_{w_\ell} ), \quad &\text{ if $t>0$,}
\\
&C_{|t|^{1/\ell}}(X_{w_1}, \dots, X_{w_\ell} )^{-1}, \quad &\text{ if $t<0$.}
                 \end{aligned}\right.
\label{approdo}
\end{equation}
 By standard ODE theory,  if $x$ belongs to a bounded set~$\Omega_0$,  there is $r_0>0$ so that  if $|t|\leq r_0 $   and $x\in\Omega_0$ the approximate exponential is well defined.

  Let us assume that the system $X_1,\dots, X_m$ has step $\kappa\in\N$ and introduce the family $Y_1,\dots, Y_q$ of their nested commutators of length at most~$\kappa$.  \color{black}  Define, given $ I=(i_1,\dots,i_n)\in\{1,\dots,q\}^n$, for $x\in K$ and $h\in\R^n$,  $ h  $ in a neghborhood of the origin
     \begin{equation}
 \label{hhh}E_{I,x}(h)=\expap (h_1 Y_{i_1})\cdots \expap (h_n Y_{i_n})(x).
   \end{equation}

  \begin{theorem}[Ball-box]\label{scatola}
Let $X_1,\dots, X_m$ be H\"ormander vector fields of step $\kappa$ in~$\R^n$. Fix an open bounded set $\Omega_0\subset\R^n$. Then there is $r_0>0$ such that 
for all $\eta\in\left]0,1\right[$ 
there are constants $\e_\eta<1< C_\eta$ such that 
for any     $\eta$-maximl triple $(I,x, r)$ with $r\leq r_0$ and $x\in \Omega_0$, we have
\begin{enumerate}[nosep,label=(\roman*)]
\item the map $h\mapsto E_{I,x}(h)$ is one-to-one on the box $Q_I( \e_\eta r)$ defined in~\eqref{qui};
 \item we have the inclusion 
 \begin{equation}\label{imbianchino} 
E_{I, x}(Q_I(\e_\eta r))\supseteq B_\rho(x, C_\eta^{-1}r);                                                                         \end{equation}  
 \item The Jacobian of the map $E_{I,x}$ admits the following estimate
 \[
C_\eta^{-1}|\lambda_I(x)|r^{\ell(I)}\leq \Big|\det\dfrac{\p E_{I,x}(h)}{\p h}\Big|\leq C_\eta |\lambda_I(x)|r^{\ell(I)} \quad \text{ for all $h\in Q_I(\e_\eta r )$.}
 \]
\end{enumerate} 
  \end{theorem}

Theorem \ref{scatola} has been proved and used in various regularity 
conditions in~\cite{Morbidelli98,MM,MMPotential}. In this paper we will work in the smooth case and we shall  choose always $\eta=\frac 12$, to make statements clean. 

  A first consequence of Theorem~\ref{scatola} is the following volume estimate. For all bounded set $\Omega\subset\R^n$ there is $r_0>0$ and $C>0$
such that
\begin{equation}\label{volumone} 
 C^{-1}\L^n(B(x,r))\leq    \max_{K\in\{1,\dots, q\}^n}|\lambda_K(x)| r^{\ell(K)}  \color{black}
 \leq C
\L^n(B(x,r))\quad \text{for all  $x\in\Omega$ and $r<r_0$.}
\end{equation} 
A couple of further consequences  are the estimates
\begin{equation}
 \label{upside2}
\begin{aligned}
&  \L^n(B(x,r))\leq Cr^n \quad\text{for all $x\in\Omega$ and $r<r_0$;}
\\& d(x,y)\leq C|x-y|^{1/\kappa}
\quad\text{for all $x,y\in\Omega$. }
\end{aligned}
\end{equation}
\begin{remark}\label{effeelle} 
 Let us observe  the following  consequence of Theorem \ref{scatola} and of the construction of the maps $E_{I, x}$.  Under the hypotheses of Theorem~\ref{scatola}, for any open bounded 
set $\Omega\subset\R^n$ there are $r_0>0$ and $C_0>0$ so that any pair of points $x,y\in\Omega 
$ with $|x-y|<r_0 $ can be connected with a piecewise integral curve of the 
vector fields $\pm X_1,\dots, \pm X_m$. The number of pieces is bounded by an universal algebraic constant $M$ depending on $m$ and $\kappa$, 
while each piece has   length $\leq C \r (x,y)$. (Recall that by definition we always have $\r(x,y)\leq d(x,y)$). 
\end{remark}
\color{black}

In the following remark, we go back briefly to the discussion of the introduction, concerning estimate \eqref{kiki}.
\begin{remark}\label{march}  
 Let $X_w$  be any nested commutator of length $|w|\leq\kappa$ constructed from a family $X_1,\dots, X_m$ of H\"ormander vector fields of step~$\kappa$ in $\R^n$.
 Given  a bounded set $\Omega$   
 and $\Omega_0\Supset \Omega$, there is~$r_0>0$ and  a positive $C$  such that \color{black}
\begin{equation}\label{kiki2} 
 \sup_{x\in\Omega ,\;|\tau|\leq r_0} \frac{|f(e^{\tau  X_w}x) - f(x)|}{|\tau|^{1/|w|}}\leq C \sup_{y\in \Omega_0} \sum_{j=1}^m|X_j f(y)|.
\end{equation} 
To check \eqref{kiki2}, it suffices to observe that, by definition of $\rho$, we have that   $e^{\t X_w}(x)\in B_\rho (x,  2|\t|^{1/|w|})$ for all $x\in\R^n$ and $\t\in\R$.   Note explicitly that we must 
use here the distance $\r$ defined  using commutators.
Then, by Remark~\ref{effeelle}, we conclude that if $x\in\Omega$ and $|\tau|< r_0$ for small $r_0$, 
\begin{equation*}
 |f(e^{\tau  X_w}x) - f(x)|\leq C\sup_{\Omega_0 }{\sum_j|X_jf|}\cdot |\tau|^{1/|w|}.
\end{equation*}
  Informally speaking, if we choose  $r_0$  small   enough, then $\Omega_0$ can be made a small open neighborhood of~$\ol\Omega$.  
\end{remark}

 \section{Anisotropic subelliptic estimates}\label{Sec3}

In this section we prove   inequality \eqref{equivalgo} and consequently  Theorem \ref{anisotta}. Namely we show the local estimate
\begin{equation}\label{trocco} 
 [f]
 _{W^{s,p}_d (\Omega)}\leq C \Big(\sum_{j=1}^m \| X_j u\|_{L^p(\Omega_0)}
  +  \|   u\|_{L^p(\Omega_0)}\Big),
\end{equation} 
for all $f\in C^1$.
Here $\Omega$ is 
bounded, $\Omega_0\supset\ol\Omega$,  $p\in\left[1,+\infty\right[$ and   $s\in\left]0,1\right[$.

\begin{remark}\label{oca}  Let us look inequality \eqref{trocco} in the Euclidean case. 
Roughly speaking, it says that derivatives of order $s<1$ in $L^p$ can be estimated with  derivatives of order $1$ in $L^p$. The inequality is essentially trivial. 
Namely, in order to get the estimate  
 \begin{equation*}
  \int\limits_{\substack{\Omega\times\Omega\\ |x-y|< r_0}} \frac{|u(x)-u(y)|^p}{|x-y|^{n+ps}}dxdy \leq C \int_{\Omega_0} |
   \nabla  f(x)|^p dx,
 \end{equation*}
 where $\Omega_0$ is an open neighborhood of $\Omega$ depending on $r_0$, 
 it suffices to apply the fundamental theorem of calculus and then Minkowski inequality
 \begin{equation*}
\begin{aligned}
&
\int_\Omega dx\int_0^{r_0} \frac{dh}{|h|^{n+ps} } |u(x)-u(x+h)|^p dx  =\int_0^{r_0}\frac{dh}{|h|^{n+ps} }
\int_\Omega dx \Big|\int_0^{|h|}  \Big|\nabla f\Big (x+t\frac{h}{|h|}\Big)\Big| dt\Big|^p 
\\&
\leq \int_0^{r_0} \frac{dh}{|h|^{n+ps}}\Big\{ \int_0^{|h|}dt
\Big[ \int_{\Omega} 
\Big|\nabla f\Big (x+t\frac{h}{|h|}\Big)\Big|^p dx \Big]^{1/p}\Big\}^p
\\&
\leq C\int_0^{r_0}\frac{dh}{|h|^{n+ps}}C|h|^{p}\int_{\Omega_0}|\nabla f(z)|^p dz  \leq C 
\int_{ \Omega_0}|\nabla f |^p, 
\end{aligned}   \end{equation*}
because the integral  in $dh$ converges for $s<1$.  
The argument of the   inequalities above is based on integration of $f$ along the curves $t\mapsto\gamma(t):= x+t\frac{h}{|h|}$,
which is not available in the H\"ormander setting. The  machinery of the approximate exponentials $E_{I,x}$ makes possible to prove inequality~\eqref{trocco} following the scheme of the chain of inequalities above. 
\end{remark}

\begin{proof}

 Before starting the proof, 
 we recall some useful consequences  of the ball-box Theorem for vector fields stated in Section~\ref{Sec2}.  We use below 
the statement with   $\eta=\frac 12$  of Theorem \ref{scatola} and
we let $ \wh C:=C_{1/2}$ and $\wh \e:=\e_{1/2} $.
 
   Let  $\Omega$ be a bounded open set and let $r_0>0$ be a small number so that Theorem~\ref{scatola} applies. 
  Let us define, given $x\in\Omega$ and $I\in\{1,\dots, q\}^n$, the set 
 \begin{equation*}
     M_{I, x}:=\Big\{y\in\R^n: d(x,y)  <  r_0\quad\text{ and  $(I, x, \wh C d(x,y))$ is $\frac 12$-maximal}\Big  \}.
  \end{equation*}
 It is easy to see that the set $     M_{I, x}$ is a metric annulus of the form
 \begin{equation*}
     M_{I, x}=\Big\{y\in\R^n: r_{I,x}<
     d(x,y)<R_{I,x} \Big  \}.
 \end{equation*}
 We have defined for all $I\in\{1,\dots,q\}^n$ and $x$ an open set  
 $M_{I, x}$ which can be empty. Furthermore, different  choices of $I$ can give 
 overlapping annuli.   Finally,   $\cup_IM_{I,x}=B(x,r_0)$. 
 The radius $R_{I,x}$ satisfies also the condition
\begin{equation*}
|\lambda_I(x)| (\wh C R_{I, x})^{\ell(I)} \geq 
\frac 12 \max_{ K\in\{1,\dots,q\}^n} 
|\lambda_K(x)| (\wh C R_{I, x})^{\ell(K)}.
 \end{equation*}
Therefore, Theorem \ref{scatola} gives the estimates
\begin{align}
&\label{laprim} \frac{1}{\wh C}|\lambda_I(x)| \leq \Big|\det\frac{\p E_{I,x}(h)}{\p h}\Big| \leq
\wh C |\lambda_I(x)|\qquad \text{if $\| h\|_I\leq \wh\e\wh C R_{I,x}$}
\\& B(x,R_{I,x})\subset E_{I,x}(Q_I(\wh C\wh\e R_{I,x}))
\\& h\mapsto E_{I,x}(h) \text{ is one-to-one on $Q_I(\wh C\wh\e R_{I,x})$.}
\end{align}
Recall again that $Q_I(\wh C\wh\e R_{I,x}):= \{h\in\R^n:\| h\|_I\leq \wh\e\wh C R_{I,x}\}$.
On the set $M_{I,x}$ we also have a lower estimate on the distance $d(x,y)$ in terms of the variable $h$. To get
this bound, let us consider 
a point $y\in M_{I,x}$ and  choose any number $\r \in \left]d(x,y),R_{I,x}\right[$. For any such choice of $\r$, we have obviously $y\in B(x,\r)$. Furthermore, 
the triple $(I, x, \wt C\rho)$ is $1/2$-maximal. 
Then we  can write  $y= E_{I,x}(h)$  for a unique $h=E_{I,x}^{-1}(y)\in 
Q_I(\wh C\wh \e \r)
$.
Therefore, given $y\in B(x,\rho)$, the unique $h$ satisfying $E_{I,x}(h)=y$ belongs to $Q_I(\wh C\wh \e \r)$, i.e. satisfies $\|h\|_I\leq \wh C\wh \e\r$. Thus,
\begin{equation*}
\|h\|_I=\|E_{I,x}^{-1}(y) \|\leq \wt C\wt \e\r, 
\quad \text{for all $y\in M_{I,x}$ and   $\rho\in\left]d(x,y), R_{I,x}\right[=
\left]d(x,E_{I,x}(h)), R_{I,x}\right[$.}
\end{equation*}
Letting $\rho\searrow  d(x,E_{I,x}(h))$, we get the useful lower bound
\begin{equation}\label{lowerbound}
 d(x,E_{I,x}(h))\geq (\wh C\wh \e)^{-1}\|h\|_{I}\quad\text{for all $h\in E_{I,x}^{-1}(M_{I,x})$.}
\end{equation}

Let now  $p\in\left[1,+\infty\right[$ and  $s\in\left]0,1\right[$. Then  
we start with the estimate  
  \[
\begin{aligned}
 \int_\Omega dx & \int_{B(x,r_0)\cap \Omega}\frac{|f(x)-f(y)|^p}{d(x,y)^{ps}\L^n(B(x,d(x,y)))}dy
\\&\leq \sum_{I\in\{1,\dots, q\}^n} \int_\Omega dx\int_{M_{I,x}\cap\Omega} \frac{|f(x)-f(y)|^p}{d(x,y)^{ps}\L^n(B(x,d(x,y))) }dy.
\end{aligned}
\]
By  the change of variable $y=E_{I,x}(h) \in M_{I,x}\cap\Omega $, 
the last integral is 
\[ 
\sum_I \int_\Omega dx\int_{E_{I,x}^{-1}(M_{I,x}\cap\Omega)}\frac{|f(x)-f(E_{I,x}(h))|^p
}{
d(x,E_{I,x}(h))^{ps} \L^n(B(x,d(x,E_{I,x}(h))))}  \Big|\det\frac{\p E_{I,x}(h)}{\p h}\Big|  \,dh
\]
\[
\leq C \sum_I \int_\Omega dx\int_{E_{I,x}^{-1}(M_{I,x}\cap\Omega)}\frac{|f(x)-f(E_{I,x}(h))|^p}{\|h\|^{\ell(I)+ps}_I} dh=:(*)
\]
 where we have used the  equivalence $\L^n(B(x,r))\simeq\max_K|\lambda_K(x)|r^{\ell(K)}$, the 
$\frac12$-maximality of $I$ and estimates~\eqref{lowerbound} and~\eqref{laprim}. Note also that $E_{I,x}^{-1}(M_{I,x}\cap\Omega)\subset Q_I(\wh C\wh \e R_{I,x})\subset Q_I(\wh C\wh \e r_0)$.    In the previous chain of inequalities we  denoted $E_{I,x}^{-1}=\Big(E_{I,x}\big|_{Q_{I }(\wh C\wh\e R_{I,x })}\Big)^{-1}$ 
\color{black}

Next recall that we can write in an obvious way
\begin{equation}
E_{I,x}(h) =\expap(h_1 Y_{i_1})\cdots \expap(h_n Y_{i_n}) =:\gamma_{I,x,h}(T_I(h)),	
\end{equation}
where the curve $t\mapsto\gamma_{I,x,h}(t)$ is parametrized on the interval $[0, T_I(h)]$ and by construction of $\expap$ it
 is a concatenation of integral curves of the vector fields $X_1,\dots, X_m$. Furthermore $|T_I(h)|\leq C\|h\|_I$ for some absolute constant~$C$. The map $x\mapsto \gamma_{I,x,t}$ is a change of variable by classical properties of flows of ODEs 
(see the discussion in~\cite{LanconelliMorbidelli})  and 
\[
C^{-1}\leq \Big| \det\frac{\partial  \gamma_{I,x,h}(t)}{\p x} \Big| \leq C,
\]
uniformly in $x$ on compact sets, $|h|< r_0$ and $t\in [0, T_I(h)]$.

Then, letting $|Xf|=|(X_1f, \dots, X_m f)|$ we get 
\[
\begin{aligned}
(*)&\leq \sum_I\int_\Omega dx\int_{Q_I(\wh C\wh \e R_{I,x})}\frac{dh}{\|h\|_I^{\ell(I)+ps}}
\Big|    \int_0^{T_I(h)}    |Xf(\gamma_{I,x,h}(t))|dt\Big|^p
\\ &
\leq C\sum_I \int_{Q_I(\wh C\wh \e r_0 )} \frac{dh}{\|h\|_I^{\ell(I)+ps}}\int_\Omega dx
\Big|    \int_0^{T_I(h)}    |Xf(\gamma_{I,x,h}(t))|dt\Big|^p
\\&
\leq C
\sum_I \int_{Q_I(\wh C\wh \e r_0 )} \frac{dh}{\|h\|_I^{\ell(I)+ps}}
\Big\{  
\int_0^{T_I(h)}  dt \Big[\int_\Omega 
          |Xf(\gamma_{I,x,h}(t))|^p dx
  \Big]^{1/p}  
\Big\}^p
\\&
\leq C\sum_I \int_{Q_I(\wh C\wh \e r_0 )} \frac{dh}{\|h\|_I^{\ell(I)+ps}}\|h\|_I^p \int_{\Omega_0}|X f(z)|^p dz
=C\int_{\Omega_0}|X f|^p,  
\end{aligned}
\]
as required.
We have used the  fact that for all $I\in\{1,2,\dots,q\}^n$,
\[
\int_{\{\|h\|_I<1\}}\frac{dh}{\|h\|_I^{\ell(I)+ps-p}}<\infty , \text{ for all $s<1$,}
\]
which can be proved by a standard decomposition as a disjoint union of  sets of the form $ \{2^{- k }<\|h\|_I\leq 2^{-k+1}\}$ with $k\in\N$.
 \end{proof}
 
\section{Estimates along commutators}\label{Sec4}

In this section we prove Theorem \ref{parolone}. Namely given H\"ormander vector fields of step~$\kappa$, $s<1$ and $1\leq p<\infty$, we show the estimate 
  \begin{equation}\label{hojo} 
 \int_0^{r_0}\frac{dt}{|t|^{1+ps/|w|}}\int_\Omega 
  \big|f(e^{tX_w}x)-f(x)\big|^p dx
\leq   C\Big( \sum_{j=1}^m \|X_j f\|_{L^p(  \Omega_0)} +\|  f\|_{L^p(  \Omega_0)} 
\Big)^p, 
  \end{equation} 
for any nested commutator
 $X_w$ with $|w|\leq \kappa$.  
 Here, as usual $\Omega$  is a bounded open set, $r_0$ is a suitable small constant and $\Omega_0\Supset \Omega$ is an enlarged set.   As we already observed, we do not reach the optimal exponent  $s=1$.
 An estimate in the same spirit was proved   by Franchi and Lanconelli
in \cite[Theorem]{FranchiLanconelli84} for diagonal vector fields.

The proof of the inequality~\eqref{hojo} is an immediate consequence  of the  Lemma~\ref{pippirillo}   
and  of Theorem~\ref{anisotta}.  
\begin{lemma}\label{pippirillo} Let $p\geq 1$. Given H\"ormander vector fields $X_1,\dots, X_m$
of step~$\kappa$, for any nested commutator~$X_w$ with $|w|\leq \kappa$,  for each $s\in\left]0,1\right[$ and 
for any $x_0\in\R^n$ there is a neighborhood $\Omega $ of $x_0$ in $\R^n$ such that given $\wt \Omega \Supset  \Omega$  there is $r_0>0$ and $C>0$ such that \color{black}
\begin{equation}\label{hoho} 
\int_0^{r_0}\frac{dt}{|t|^{1+ps/|w|}}\int_\Omega 
  \big|f(e^{tX_w}x)-f(x)\big|^p dx\leq
  C \int\limits_{\wt\Omega}dx \int\limits_{\wt\Omega}dy \frac{|f(x)-f(y)|^p}{d(x,y)^{ps}\L^n(B(x, d(x,y)))}
\end{equation}
for any $C^1$ function $f$. 
\end{lemma}
This statement was  proved in \cite{Morbidelli98} 
for commutators $X_w$ of length $|w|=1$. 
Here  we provide  a sketch of the  proof   and we show that the argument works  for any commutator~$X_w$ of any length $1\leq |w|\leq \kappa$. 

The argument is based on the well known 
lifting procedure by Rothschild and Stein,  
which we now
briefly describe.
Let $X_1, \dots, X_m$ be vector fields of step $\kappa$ at any point of $\R^n$.
Let us fix a point $x_0\in\R^n$.  In \cite{RothschildStein} it was proved that there exists a neighborhood  $U\times V$ of $(x_0,0)\in \R^n\times \R^d=:\R^N$ such that on $U\times V\ni(x,\tau)$ we can define new vector fields 
\begin{equation}\label{liftati}
\wt X_j= X_j+\sum_{\b=1}^d a_{j,\beta}(x,\tau)\frac{\p}{\p  \tau_\beta}, \quad\text{where $(x,\tau)\in U\times V$,}
\end{equation}
which are free up to order $\kappa$ in $U\times V$. This means that the only linear relations among commutators of order $\leq \kappa$ of the vector fields $\wt X_j$ have constant coefficients in $U\times V$ and are given by the antisymmetry and the Jacobi identity. Note also that $N$ is the dimension of the    nilpotent  
free  Lie algebra of step $\kappa$ with $m$ generators.  
 However, usually the Lie algebra generated by $\wt X_1,\dots, \wt X_m$ is not nilpotent.

Let us go back to the family $Y_1 , \dots, Y_q $ introduced in the previous section as an enumeration of the family $X_w$ as $1\leq |w|\leq \kappa$. Note incidentally that it is
$q>N$ for vector fields of step $\geq 2$.
For a given $Y_k =X_{w }$ with 
$w=w(k)= w_1w_2 \dots  w_\ell $ belonging to such family we define the lifted commutator $\wt Y_j =[\wt X_{w_1}, [\wt X_{w_2}, \dots ,   [\wt X_{w_{\ell-1}}, \wt X_{w_\ell}] \dots]] $.

Up to reordering the commutators  $\wt Y_1,\dots, \wt Y_q$ we can assume that the subfamily of the first $N$ commutators, 
\begin{equation}\label{gatto}
\wt Y_1=\wt X_1, \wt Y_2= \wt X_2, \dots, \wt Y_m=\wt X_m, \wt Y_{m+1},\dots, \wt Y_{N}
\end{equation}
is linearly independent. 
The lifted vector fields define a distance $\wt d$ whose properties are  established in \cite[Lemma~3.2]{NagelSteinWainger} and \cite[Lemma~4.4]{Jerison}. Then we get the
 following lemma.
\begin{lemma}[See { \cite[Lemma 4.3]{Morbidelli98}}]\label{quattrotre} Let $x_0\in \R^n$ and let $U$ and $V$ be the sets arising in the  lifting procedure.
Given   compact sets $E\subset U$ and $ H\subset  V$, there is $\delta_0 > 0 $ and $C>0$ 
 such that for all $x,y\in E$ with $d(x,y)<\delta_0$, we have 
\begin{equation}
\int\limits_{
\{(\tau,\sigma)\in H\times H\,:\, \wt d((x,\tau),(y,\sigma))\leq \delta_0\}}\frac{ d\tau d\sigma }{\wt d((x,\tau), (y,\sigma))^{Q+ps}}\leq
C\frac{1}{d(x,y)^{ps} \L^n(B(x, d(x,y)))}
\end{equation}
where $Q=\ell(\wt Y_1)+\dots + \ell(\wt Y_N)$ is the homogeneous dimension  of the free Carnot     group of step~$\kappa$ with 
$m $ generators. \footnote{In our notation,  
the free Carnot     group of step~$\kappa$ with 
$m $ generators has topological dimension $N$. The homogeneous dimension $Q$ has the property that $\L^N(  B_r) = C r^Q$, where $B_r$ is a Carnot--Carath\'eodory ball of radius~$r>0$ centered at any point.
} 
\color{black}
\end{lemma}

 We do not include the proof of Lemma \ref{quattrotre}, whose
 argument is based on \cite[Lemma~3.2]{NagelSteinWainger}   or \cite[Lemma~4.4]{Jerison}. See \cite[Lemma 4.3]{Morbidelli98} for the details.

Then we have the following Lemma.

\begin{lemma}\label{quarantasei} 
Let  $\wt Y_1, \dots,   \wt Y_q$  \color{black} be the commutators  introduced in \eqref{gatto} and let $U\times V\subset \R^n\times\R^d=\R^N$ be the sets appearing in \eqref{liftati}. 
Then, given  $s\in\left]0,1\right[$ and $p \in\left[1,+\infty\right[$, fixed  open sets $G\Subset  \wt G\Subset U\times V$,
 and given a commutator~$X_w$ with $|w|\leq \kappa$, there   are $\d_0>0$ and   $ \alpha>0$ such that  
 \begin{equation*}
  \int_G d\xi \int\limits_{\{|t|< \delta_0^{|w|}  , \;e^{\alpha t X_w}(\xi)\in G\}}
 \frac{dt}{|t|^{1+ps/|w|}}|u(e^{t\alpha X_w}(\xi))-u(\xi)|^p
 \leq C\int\limits_{\wt G\times\wt G} \frac{|u(\xi)-u(\eta)|^p d\xi d\eta}{\wt d(\xi,\eta)^{Q+ps}}.\end{equation*} 
\end{lemma}  
In the statement we denoted with $\xi=(x,\tau)$ variables in the lifted space
$\R^n\times\R^d$.

\begin{proof}
 [Sketch of the proof of Lemma~\ref{quarantasei}] The argument is similar to the one appearing in \cite[Lemma~4.4]{Morbidelli98}.     We sketch it, because there are some slight differences of notation and because here we consider commutators of any step. 
 First of all, 
 in view of the discussion in \cite[p.~272]{RothschildStein},
we can rearrange the choice of the basis $\wt Y_1, \dots, \wt Y_N$ in~\eqref{gatto}  in such a way that    $\wt X_w = 
 \wt Y_j$ for some $j\in \{1,\dots, N\}$ with $\ell_j=|w|$. Denote by $Q:=\sum_{j=1}^N \ell_j$ the homogeneous dimension.
 Fix then an open set $G^*$ such that $G\Subset G^*\Subset \wt G$ and define the exponential map $\wt\Phi_\xi(h):=\exp(h_1\wt Y_1+\cdots + h_N\wt Y_{N})(\xi)$. Taking $\d_0$ small enough we may assume that $\wt\Phi_\xi(h)\in G^*$  if $\|h\|:=\max_{j\leq N}|h_j|^{1/\ell_j}<\delta_0 $. In particular, if $\alpha\in[0,1]$ we also have   $e^{\alpha h_j \wt Y_j}\xi\in G^*$,  if $\|h\| <\delta_0 $.
\color{black} 
  Let us start from the inequality~(44) in~\cite{Morbidelli98},
 which reads as 
 \begin{equation*}
\begin{aligned}
 \int_G d\xi &
 \int\limits_{\{|t|< \delta_0^{|w|}  , \;e^{\alpha t \wt X_w}(\xi)\in G\}}
 \frac{dt}{|t|^{1+ps/|w|}}|u(e^{t\alpha \wt X_w}(\xi))-u(\xi)|^p
 \\&
 = \int_G d\xi
 \int\limits_{\{|h_j| < \delta_0^{\ell_j }  , \;e^{\alpha h_j \wt Y_j}(\xi)\in  G \color{black}\}}
 \frac{dh_j}{|h_j|^{1+ps/|w|}}|u(e^{h_j\alpha \wt Y_j} (\xi))-u(\xi)|^p
 \\&
\leq C 
 \int_G d\xi \int_{\|h\|\leq \delta_0}  
 \frac{  dh}{ \|h\|^{Q+ps}} |f(e^{h_j\a\wt Y_j}\xi)- f(\xi)|^p =:(*).
\end{aligned}
 \end{equation*}
 We used the equivalence 
 \begin{equation}
  \int_{|h_j|\leq \delta_0^{\ell_j}}\frac{\psi(h_j) dh_j}{|h_j|^{1+ps/\ell_j}}\simeq 
  \int_{\|h\|<\delta_0}\frac{\psi(h_j) dh}{\|h \|^{Q+ps}},
 \end{equation} 
 valid for all $\psi=\psi(h_j)$ nonnegative and measurable, with $h=(h_1, \dots, h_N)$, $\|h\|:=\max_{k=1,\dots, N}|h_k|^{1/\ell_k}$. See \cite[(40) and~(41)]{Morbidelli98}.
 To conclude the proof, introducing the exponential map $\wt\Phi_\xi(h)=e^{h_1 \wt Y_1 +\cdots+
 h_N\wt Y_N}(\xi)$, by the triangle inequality, we have 
 \begin{equation*}
\begin{aligned}
(*)\leq \int_G d\xi &\int_{\|h\|\leq \delta_0}  
 \frac{  dh}{ \|h\|^{Q+ps}} |f(\wt\Phi_\xi(h))- f(\xi)|^p 
 \\+
 &\int_G d\xi  \int_{\|h\|\leq \delta_0}  
 \frac{  dh}{ \|h\|^{Q+ps}} |f(\wt\Phi_\xi(h))-f(e^{h_j\a\wt Y_j}\xi) |^p .
\end{aligned}
 \end{equation*}
To estimate the first term we just use the change of variable $h\mapsto  \wt\Phi_\xi(h) =:\eta$, which is nonsingular because the vector fields $\wt Y_1,\dots, \wt Y_N$ are linearly independent. 
 To estimate the second one, we must choose a sufficiently small  $\alpha>0$ so that, roughy speaking, $e^{h_j\a\wt Y_j}(\xi)$ stays rather close to  $\xi$ and the second term admits an analogous estimate
 \begin{equation}
  (*)\leq \int_{\wt G\times\wt G}
  \frac{|u(\xi)-u(\eta)|^p}{\wt d(\xi,\eta)^{Q+ps}}d\xi d\eta.
 \end{equation} 
 See~\cite{Morbidelli98} for a detailed explaination.  
\end{proof}

\begin{proof}[Proof of Lemma \ref{pippirillo}] We follow the argument of the proof of Proposition~4.2 in~\cite[p.~237-238]{Morbidelli98}. Let $X_1, \dots, X_m$ be a family of H\"ormander vector fields of step~$\kappa\in\N$. Let us choose  a 
commutator $X_w $ with length~$|w|\leq \kappa$. Fix $x_0\in\R^n$ and introduce the vector fields $\wt X_1,\dots, \wt X_m$ on the sets $U\times V\subset\R^N $ appearing in~\eqref{liftati} and~\eqref{gatto}. Here $N$ and $Q$ are the topological and homogeneous dimension of the free Carnot group of step $\kappa$ with $m$ generators. By properties of free Lie algebras, see \cite[p.~272]{RothschildStein}, we may choose the vector fields $\wt Y_1,\dots, \wt Y_N$ in~\eqref{gatto}  in the family $\wt Y_1,\dots \wt   Y_q$ assuming that $\wt X_w =\wt  Y_j$ for some $j\leq N$ with $\ell_j=|w|$.

We apply first Lemma~\ref{quarantasei} to a set of the form $G=O\times H\Subset  \wt O\times\wt H\Subset U\times V$, where $\wt H$ is a small open neighborhood of the origin in $\R^{N-n}$. Recall also that the function $f$ does not depend on the additional variables in $V$. Then we get the  inequality
\begin{equation*}
\begin{aligned}
 &\int\limits_O dx\int\limits_{  
 |t|<\delta_0^{|w|},\;    e^{\alpha tX_w}x\in O  
 }\frac{dt}{|t|^{1+ps/|w|}}|f(e^{\alpha t X_w}(x))-u(x)|^p
 \\ &\qquad \leq
 C\int\limits_{ \wt O\times\wt H   }  dxd\t \int\limits_{ \wt O\times\wt H   } dyd\s \frac{|f(x )-f(y )|^p  }{\wt d((x,\t),(y,\s))^{Q+ps}}
 \\&\qquad 
 = C
 \int\limits_{ \wt O\times\wt O    }  dxdy |f(x )-f(y )|^p
 \int\limits_{ \wt H\times\wt H   }
 d\t d\s \frac{1  }{\wt d((x,\t),(y,\s))^{Q+ps}}
\\&
\qquad \leq C \int\limits_{ \wt O\times\wt O    }  
dxdy \frac{|f(x )-f(y )|^p}{d(x,y)^{ps}\L^n(B(x, d(x,y)))},
 \end{aligned}
\end{equation*}
by Lemma~\ref{quattrotre}.  

Covering any given open bounded $\Omega$ set with a finite family of open sets $O$ of the form 
appearing in the discussion above, we  obtain the proof of  the    inequality~\eqref{hoho} on~$\Omega$.
\end{proof}

\footnotesize

\def\cprime{$'$} \def\cprime{$'$}
\providecommand{\bysame}{\leavevmode\hbox to3em{\hrulefill}\thinspace}
\providecommand{\MR}{\relax\ifhmode\unskip\space\fi MR }
\providecommand{\MRhref}[2]{%
  \href{http://www.ams.org/mathscinet-getitem?mr=#1}{#2}
}
\providecommand{\href}[2]{#2}

\end{document}